\documentclass[12pt,a4paper]{amsart}
\usepackage{amsmath, amssymb, amsfonts,enumerate}
\usepackage[all]{xy}
\usepackage{amscd}

\newtheorem{theorem}{Theorem}[section]
\newtheorem{lemma}[theorem]{Lemma}
\newtheorem{corollary}[theorem]{Corollary}
\newtheorem{proposition}[theorem]{Proposition}

\newtheorem*{question*}{Question}

\theoremstyle{definition}

\newtheorem*{definition*}{Definition}

\theoremstyle{remark}

\newtheorem*{remark*}{Remark}

\newtheorem{example}{Example}[section]

\numberwithin{equation}{section}

\newcommand {\K}{\mathbb{K}} 
\newcommand {\N}{\mathbb{N}} 
\newcommand {\Z}{\mathbb{Z}} 

\newcommand{\FF}{\mathcal{F}}
\newcommand{\PP}{\mathcal{P}}

\DeclareMathOperator{\mdim}{mdim}

\begin{document}
\title{A Garden of Eden theorem for linear subshifts}

\author{Tullio Ceccherini-Silberstein}
\address{Dipartimento di Ingegneria, Universit\`a del Sannio, C.so
Garibaldi 107, 82100 Benevento, Italy}
\email{tceccher@mat.uniroma1.it}
\author{Michel Coornaert}
\address{Institut de Recherche Math\'ematique Avanc\'ee,
                                            Universit\'e  de Strasbourg,
                                                 7 rue Ren\'e-Descartes,
                                               67084 Strasbourg Cedex, France  }
\email{coornaert@math.u-strasbg.fr}
 \begin{abstract}
Let $G$ be an amenable group and let $V$ be a finite-dimensional vector space over an arbitrary field $\K$. 
  We prove that if $X \subset V^G$ is a strongly irreducible linear subshift of finite type and 
  $\tau \colon X \to X$ is a linear cellular automaton, then $\tau$ is surjective if and only if it is 
  pre-injective.
We also prove that if $G$ is countable and $X \subset V^G$ is a strongly irreducible linear subshift,
then every injective linear cellular automaton $\tau \colon X \to X$ is surjective. 
     \end{abstract}

\maketitle

\section{Introduction}
\label{sec:introduction}

The goal of this article is to give a version of the Moore-Myhill Garden of Eden theorem 
for linear cellular automata  defined over   certain   linear subshifts.
Before stating our main results, let us briefly recall some basic notions from symbolic dynamics.
 \par 
Consider a group $G$ and a set $A$.
The set $A^G  = \{x  \colon G \to A\}$ is called the set of \emph{configurations} over the group $G$ and the \emph{alphabet} $A$.
We equip $A^G = \prod_{g \in G} A$ with its \emph{prodiscrete} topology, i.e., with the product topology obtained by taking the discrete topology on each factor $A$ of $A^G$. 
We also endow $A^G$ with
the left action of $G$ defined by $gx(h) = x(g^{-1}h)$ for all 
$g,h \in G$ and $x \in A^G$.
This action is continuous with respect to the prodiscrete topology and is called the $G$-\emph{shift} action on $A^G$. It is customary to refer to  the $G$-space $A^G$ as the
\emph{full shift} over the group $G$ and the alphabet $A$.
\par
A closed $G$-invariant subset of $A^G$ is called a \emph{subshift}.
\par
For $x \in A^G$ and $\Omega \subset G$, let $x\vert_\Omega$ denote the restriction of $x$ 
to $\Omega$.
One says that a subshift $X \subset A^G$ is \emph{irreducible} if,
for every finite subset $\Omega \subset G$ and any two configurations  $x_1$ and $x_2$ in  $X$, there exist an element $g \in G$ and a configuration $x \in X$ such that 
$x\vert_\Omega = x_1\vert_\Omega$ and $x\vert_{g\Omega} = x_2\vert_{g\Omega}$.   
\par
Given a finite subset $\Delta \subset G$, one says that a
subshift $X \subset A^G$ is $\Delta$-irreducible if the following condition is satisfied:
if $\Omega_1$ and $\Omega_2$ are finite subsets of $G$ such that there is no element 
$g \in \Omega_2$ such that the set $g\Delta$ meets $\Omega_1$ then, given any two configurations $x_1,x_2 \in X$, there exists a configuration $x \in X$ such that
$x\vert_{\Omega_1} = x_1\vert_{\Omega_1}$  and $x\vert_{\Omega_2} = x_2\vert_{\Omega_2}$. 
A subshift $X \subset A^G$ is said to be \emph{strongly irreducible} if there exists a finite subset $\Delta \subset G$ such that $X$ is $\Delta$-irreducible  (cf. \cite[Definition 4.1]{fiorenzi-strongly}).
Note that if $G$ is infinite then every strongly irreducible subshift is irreducible.
A trivial example of a strongly irreducible subshift is provided by the full shift $A^G$ which is 
$\Delta$-irreducible for $\Delta = \{1_G\}$.
\par
If $D$ is a finite subset of $G$ and $L$ is a subset of $A^D$,   then
\begin{equation}
\label{e:def-X-d-P}
X_G(D,L) = \{ x \in A^G : (g^{-1}x)\vert_D \in L \text{ for all } g \in G \}
\end{equation}
is clearly a subshift of $A^G$.
A subshift $X \subset A^G$ is said to be \emph{of finite type} if there exist a finite subset $D \subset G$ and a subset $L \subset A^D$ such that $X = X_G(D,L)$.
 One then says that the finite subset $D \subset G$ is a \emph{defining window} and that
  $L \subset A^D$ is a \emph{defining law}, relative to the defining window $D$, for the subshift $X$.
  Note that the full shift $A^G$ is a subshift of finite type of itself admitting $D = \{1_G\}$ as a defining window with defining   law $L = A^D \cong A$. 
  \par
 A map $\tau \colon X \to Y$ between subshifts $X,Y \subset A^G$ is called a \emph{cellular automaton} if there exist a finite subset $M \subset G$ 
and a map $\mu \colon  A^M \to A$ such that 
\begin{equation} 
\label{e;local-property}
\tau(x)(g) = \mu( (g^{-1}x)\vert_M)  \quad  \text{for all } x \in X \text{ and } g \in G.
\end{equation}
Such a set $M$ is then called a \emph{memory set} and $\mu$ is called a \emph{local defining map} for $\tau$. It immediately follows from the above definition that every cellular automaton $\tau \colon X \to Y$ is continuous and $G$-equivariant.
\par 
If $\tau \colon X \to X$ is a cellular automaton from a subshift $X \subset A^G$ into itself, a configuration 
$x_0 \in X$ is called a 
\emph{Garden of Eden configuration} for $\tau$ if $x_0$ is not in the image of $\tau$.
The origin of this biblical terminology comes from the fact that a configuration $x_0 \in X$ is a Garden of Eden configuration for $\tau$ if and only if,
whatever the choice  of an initial configuration $x \in X$, the sequence of its iterates $x, \tau(x), \tau^2(x), \dots,\tau^n(x),\dots$ can only take the value $x_0$ at time $n = 0$.
\par
Two configurations in $A^G$ are said to be \emph{almost equal} if they coincide outside a finite subset of $G$.
  One says that a cellular automaton $\tau \colon X \to Y$ between subshifts $X,Y \subset A^G$ is 
  \emph{pre-injective} if whenever two configurations $x_1,x_2 \in X$ are almost equal and
 satisfy $\tau(x_1) = \tau(x_2)$ then one has $ x_1 = x_2$.
Injectivity clearly implies pre-injectivity but there are pre-injective cellular automata which are not injective.
\par
The classical \emph{Garden of Eden theorem} \cite{ceccherini} states that if $\tau \colon A^G \to A^G$   is a cellular automaton defined on the full shift over an amenable group $G$ and a finite alphabet $A$, 
then the surjectivity of   $\tau$ (i.e.,  the absence of Garden of Eden configurations for $\tau$)
is equivalent to its  pre-injectivity  (see Subsection \ref{ss;amen} for the definition of amenability).
\par
The Garden of Eden theorem was extended by F.~Fiorenzi
to cellular automata $\tau \colon X \to X$ for subshifts $X \subset A^G$ with $A$ finite in the following two cases:
(1) $G = \Z$ and $X \subset A^\Z$ is an irreducible subshift of finite type \cite[Corollary 2.19]{fiorenzi-sofic};
(2)  $G$ is a finitely generated amenable group and $X \subset A^G$ is a strongly irreducible subshift of finite type \cite[Corollary 4.8]{fiorenzi-strongly}. 
\par
Now let $G$ be a group, $\K$  a field, and $V$ a vector space over $\K$.
Then there is a natural product vector space structure on $V^G$ and
the shift action of $G$ on $V^G$ is clearly $\K$-linear with respect to this vector space structure.
One says that a subshift $X \subset V^G$ is a \emph{linear subshift}
if $X$ is a  vector subspace of $V^G$.
Given linear subshifts $X,Y \subset V^G$,
a cellular automaton $\tau \colon X \to Y$ is called a \emph{linear cellular automaton} if the map $\tau$ is $\K$-linear.
Note that if $X,Y \subset V^G$ are linear subshifts and $\tau \colon X \to Y$ is a linear cellular automaton, then the pre-injectivity of $\tau$ is equivalent to the fact that the zero configuration 
is the unique configuration with finite support lying in the kernel of $\tau$.
 \par 
In \cite[Theorem 1.2]{garden} and \cite[Corollary 1.4]{induction}, 
we proved the following linear version of the Garden of Eden theorem:

\begin{theorem}
\label{t:goe-lca-full-shift}
 Let $G$ be an amenable group, $\K$ a field, and $V$ a finite-dimensional  vector space over $\K$.
Let  $\tau \colon V^G \to V^G$ be a linear cellular automaton.
Then $\tau$ is surjective if and only if it is pre-injective. 
\end{theorem}   

 \par
The main result of the present paper is the following:
 
\begin{theorem}
\label{t;main}
Let $G$ be an amenable group, $\K$ a field, and $V$ a finite-dimensional  vector space over $\K$.
Let $X \subset V^G$ be a strongly irreducible linear subshift of finite type and let $\tau \colon X \to X$ be a linear cellular automaton.
Then $\tau$ is surjective if and only if it is pre-injective. 
\end{theorem}

Note that Theorem \ref{t:goe-lca-full-shift} may be recovered from Theorem \ref{t;main} by taking 
$X = V^G$.
\par
A group $G$ is said to be \emph{surjunctive} \cite{gottschalk} if, for any finite
alphabet $A$, every injective cellular automaton $\tau \colon A^G \to A^G$ over $G$ is surjective.
It was shown by W. Lawton (cf. \cite{gottschalk}) that all residually finite groups are surjunctive.
On the other hand, as injectivity implies pre-injectivity, it immediately follows from the Garden of Eden Theorem \cite{ceccherini} that all amenable groups are surjunctive.
More generally,  M. Gromov \cite{gromov-esav} and B. Weiss \cite{weiss-sgds} proved that all 
\emph{sofic} groups are surjunctive.
The class of sofic groups includes in particular all residually amenable groups and therefore all residually finite groups as well as all amenable groups. 
As far as we know, there is no example of a non-surjunctive nor  even of a non-sofic group in the literature up to now.
\par
By analogy with the classical finite alphabet case, the following definition was introduced in \cite[Definition 1.1]{israel}.
A group $G$ is said to be \emph{L-surjunctive} if, for any field
$\K$ and any finite-dimensional vector space $V$ over
$\K$, every injective linear cellular automaton 
$\tau \colon V^G \to V^G$ is surjective.
It turns out (see \cite{israel}) that a group $G$ is L-surjunctive
if and only if $G$ satisfies Kaplansky's stable finiteness conjecture, that is,    the group algebra $\K[G]$ is stably finite for any field $\K$ (recall that a ring $R$ is said to be \emph{stably finite} if every one-sided invertible square matrix over $R$ is also two-sided invertible).
A linear analogue  of the Gromov-Weiss theorem, namely that all sofic groups are L-surjunctive,
was established in  \cite[Theorem 1.2]{israel}.
From this result we deduced that sofic groups satisfy the Kaplansky conjecture on the stable finiteness of group algebras, a result previously established -- with completely different methods
involving embeddings of the group rings into continuous von Neumann regular rings -- by G. Elek and 
A. Szab\'o \cite{es-direct}.
\par
Now, given a group $G$ and a vector space $V$ over a field $\K$, let us  say that a linear subshift 
$X \subset V^G$ is \emph{L-surjunctive} if every injective linear cellular automaton $\tau \colon X \to X$ is surjective.
An immediate  consequence of Theorem \ref{t;main} 
is the following:

 \begin{corollary}
\label{c:strongly-irr-flt-are-L-surj}
Let $G$ be an amenable group, $\K$ a field, and $V$ a finite-dimensional vector space over $\K$.
Then every strongly irreducible linear subshift of finite type $X \subset V^G$ is L-surjunctive. 
\end{corollary} 

In fact, when the amenable group $G$ is assumed to be countable, we can remove the hypothesis that the subshift $X$ is of finite type in the previous statement so that we get: 

 \begin{theorem}
\label{t:strongly-irr-are-L-surj}
Let $G$ be a countable amenable group, $\K$ a field, and $V$ a finite-dimensional vector space over $\K$.
Then every strongly irreducible linear subshift $X \subset V^G$   is L-surjunctive. 
\end{theorem} 

We do not know whether Theorem \ref{t:strongly-irr-are-L-surj} remains true if the countability assumption is removed.
\par
The paper is organized as follows.
Section \ref{sec:preliminaries} contains the necessary preliminaries and background material.
We recall in particular the definition and main properties of mean dimension for vector subspaces of configurations $X \subset V^G$, where $G$ is an amenable group and $V$ a finite-dimensional vector space.
In Section \ref{sec:mean-dim-strongly-irr} we study mean dimension of strongly irreducible linear subshifts. 
We  prove  that if $X \subset V^G$ is a strongly irreducible linear subshift  then the  mean dimension of $X$ is greater than the mean dimension of any proper linear subshift $Y \subsetneqq X$ 
(Proposition \ref{p;mdimY<mdimX}).
This result implies in particular that every nonzero strongly irreducible linear subshift has positive mean dimension (Corollary \ref{cor:positive-mdim-if-strongly}).
In Section \ref{sec:Mittag-Leffler-CIP} we use the Mittag-Leffler lemma for projective sequences of sets to prove that if $G$ is a countable group, $V$ a finite-dimensional vector space, and $X \subset V^G$ a linear subshift, then every linear cellular automaton
$\tau \colon X \to V^G$ has a closed image in $V^G$ for the prodiscrete topology.
This enables us to prove Theorem \ref{t:strongly-irr-are-L-surj}.
The closed image property of linear cellular automata is extended to possibly uncountable groups in Section \ref{sec:cip-sflt}
under the additional hypothesis that the source linear subshift $X$ has finite type.
 The proof of our Garden of Eden theorem (Theorem \ref{t;main}) is given in Section \ref{sec:proof-GOE}.
It consists in showing that both the surjectivity and the pre-injectivity of 
$\tau$ are equivalent to the fact that the linear subshifts $X$ and $\tau(X)$ have the same mean dimension (cf. Corollary \ref{c;deux}).
In the last two sections we describe some examples of linear cellular automata
which are either pre-injective but not surjective or surjective but not pre-injective.


\section{Preliminaries and background}
\label{sec:preliminaries}

In this section we collect some preliminaries and  background material that will be needed in the sequel.

\subsection{Neighborhoods} (see \cite[Section 2]{garden})
Let $G$ be a group. 
\par
Let $E$ and $\Omega$ be subsets of $G$. 
The   $E$-\emph{neighborhood}  of $\Omega$ is the subset $\Omega^{+E} \subset G$ consisting of all elements $g \in G$ such that
the set $gE$ meets $\Omega$. Thus, one has
$$
 \Omega^{+E} = \{g \in G : gE \cap \Omega \not= \varnothing\} = \bigcup_{e \in E}\Omega e^{-1} 
 = \Omega E^{-1}.
$$

\begin{remark*}
The definition of $\Delta$-irreducibility given in the introduction may be reformulated by saying that,
given a group $G$, a set $A$, and a finite subset $\Delta \subset G$,   a subshift $X \subset A^G$ 
is $\Delta$-\emph{irreducible} if the following condition is satisfied:
if $\Omega_1$ and $\Omega_2$ are finite subsets of $G$ such that 
$\Omega_1^{+\Delta} \cap \Omega_2 = \varnothing$ then, given any two configurations $x_1,x_2 \in X$, there exists a configuration $x \in X$ such that
$x$ coincides with $x_1$  on $\Omega_1$ and with $x_2$ on $\Omega_2$.
\end{remark*}

The following facts will be frequently used in the sequel:

\begin{proposition} 
\label{p:prop-int-ferm-bord}
Let $G$ be a group. 
Let $E$ and $\Omega$ be subsets of $G$. Then the following hold:
\begin{enumerate}[\rm (i)]
 \item 
if $1_G \in E$, then $  \Omega \subset \Omega^{+E}$;
\item 
if $g \in G$, then  $g(\Omega^{+E}) = (g\Omega)^{+E}$ so that we can omit parentheses
and simply write $g\Omega^{+E}$ instead;
\item
if $\Omega$ and $E$ are finite, then $\Omega^{+E}$ is finite. 
\end{enumerate}
\end{proposition}

\begin{proof}
This immediately follows from the definition of $\Omega^+E$.
\end{proof}

\begin{proposition} 
\label{p:tau-coinc-sur-int}
Let $G$ be a group and let $A$ be a set.
Let $\tau \colon A^G \to A^G$ be a cellular automaton with memory set $M$. 
 Suppose that there is a subset $\Omega \subset G$ and two configurations $x,x' \in A^G$ such that $x$ and $x'$ coincide on $\Omega$.   
Then the configurations $\tau(x)$ and $\tau(x')$ coincide outside $(G \setminus \Omega)^{+M}$.
\end{proposition}

\begin{proof}
It suffices to observe that \eqref{e;local-property} implies that $\tau(x)(g)$ depends only on the restriction of $x$ to $gM$.
\end{proof}

\subsection{Amenable groups}
\label{ss;amen}
(see for example \cite{greenleaf}, \cite{paterson})
A group $G$ is said to be \emph{amenable} if there exists a left-invariant finitely-additive probability measure defined on the set $\PP(G)$ of all subsets of $G$, that is, a map $m \colon \PP(G) \to [0,1]$ satisfying the following conditions:
\begin{enumerate}[({\rm A}-1)]
\item 
$m(A \cup B) = m(A) + m(B) - m(A \cap B)$ for all $A,B \in \PP(G)$ \ (finite additivity);
\item  
$m(G) = 1$ \ (normalization); 
\item 
$m(gA) = m(A)$ for all $g \in G$ and $A \in \PP(G)$ \ (left-invariance). 
\end{enumerate}
By a fundamental result of E. F\o lner \cite{folner}, a group $G$ is amenable if and only if it admits a net $\FF = (F_j)_{j \in J}$ consisting of nonempty finite subsets $F_j \subset G$ indexed by a directed set $J$ such that
\begin{equation}
\label{e:folner-net}
\lim_j \frac{\vert F_j^{+E} \setminus F_j \vert}{\vert F_j \vert} = 0 \quad \text{ for every finite subset } E \subset G,
\end{equation}
where we use  $\vert \cdot \vert$ to denote cardinality of finite sets.
Such a net $\FF$ is called a \emph{right F\o lner net} for $G$.
\par
All finite groups, all solvable groups, and all finitely generated groups of subexponential growth are amenable.
On the other hand, if a group $G$ contains a nonabelian free subgroup then $G$ is not amenable.

\subsection{Tilings} (see \cite[Section 2]{garden})
\label{ss;tilings}
Let $G$ be a group. Let $E$ and $F$ be subsets of $G$. A subset $T \subset G$ is called an $(E,F)$-\emph{tiling} if it satisfies the following conditions:
\begin{enumerate}[(T-1)]
\item 
the subsets $gE$, $g \in T$, are pairwise disjoint;
\item 
$G = \bigcup_{g \in T} gF$.
\end{enumerate}
\par
Note that if $T$ is an $(E,F)$-tiling then it is also an $(E',F')$-tiling for
all $E'$, $F'$ such that $E' \subset E$ and $F \subset F' \subset G$.

An easy consequence of Zorn's lemma is the following:

\begin{lemma} 
\label{l;tilings-exist}
Let $G$ be a group. Let $E$ be a nonempty subset of $G$ and let
$F = EE^{-1} = \{ab^{-1} : a,b \in E\}$. Then $G$ contains an $(E,F)$-tiling. 
\end{lemma}

\begin{proof}
See \cite[Lemma 2.2]{garden}.
\end{proof}

In amenable groups we shall use the following lower estimate for the asymptotic growth of tilings with respect to F\o lner nets:

\begin{lemma} 
\label{l;tiling}
Let $G$ be an amenable group and let $(F_j)_{j \in J}$ be a right F\o lner net for $G$.
Let $E$ and $F$ be finite subsets of $G$ and suppose that
$T \subset G$ is an $(E,F)$-tiling. 
For each $j \in J$, let  
$T_j $ be the subset of $T$ defined by $T_j = \{g \in T: gE \subset F_j\}$.
Then there exist a real number $\alpha> 0$ and an element $j_0 \in J$ such that 
$\vert T_j \vert \geq \alpha \vert F_j \vert$
for all $j \geq j_0$. 
\qed
\end{lemma}

\begin{proof}
See \cite[Lemma 4.3]{garden}.
\end{proof}

 \subsection{Mean dimension}
Let $G$ be an amenable group, $\FF = (F_j)_{j \in J}$ a right F\o lner net for $G$, and $V$ a 
finite-dimensional vector space over some field $\K$.
Given a subset $\Omega \subset G$, we shall denote by $\pi_\Omega \colon V^G \to V^\Omega$ the projection  map.
Observe that $\pi_\Omega$ is $\K$-linear for every $\Omega \subset G$.
Observe also that the vector space 
$V^\Omega$ is finite-dimensional  if $\Omega$ is a finite subset of $G$.
\par
Let $X$ be a vector subspace of $V^G$. The \emph{mean dimension}  
$\mdim_\FF(X)$ of $X$ with respect to the right F\o lner net $\FF$ is defined by
\begin{equation}  \label{e;mdim}
\mdim_\FF(X) = \limsup_{j} \frac{\dim(\pi_{F_j}(X))}{\vert F_j \vert},
\end{equation}
where we use  $\dim(\cdot)$ to denote dimension of finite-dimensional $\K$-vector spaces.
\par
It immediately follows from this definition that 
$\mdim_\FF(V^G) = \dim(V)$ and that $\mdim_\FF(X) \leq \mdim_\FF(Y)$ whenever $X$ and $Y$ are vector subspaces of $V^G$ such that $X \subset Y$. In particular, we have $0 \leq \mdim_\FF(X) \leq \dim(V)$ for
every vector subspace $X \subset V^G$.

An important property of linear cellular automata is the fact that 
 they cannot increase mean dimension of vector subspaces:

\begin{proposition}
\label{p;tau}
Let $G$ be an amenable group, $\FF = (F_j)_{j \in J}$ a right F\o lner net for $G$, and
$V$  a finite dimensional vector space over a field $\K$. 
Let $\tau \colon V^G \to V^G$ be a linear cellular automaton and
let $X \subset V^G$ be a vector subspace.  Then one has
$\mdim_\FF(\tau(X)) \leq \mdim_\FF(X)$.
\end{proposition}

\begin{proof}
See \cite[Proposition 4.7]{garden}.
\end{proof}

 \begin{remark*}
It may be shown that if $G$ is an amenable group, $\FF$ a right F\o lner net, $V$ a finite-dimensional vector space, and 
$X \subset V^G$ a linear subshift, then the $\limsup$ in the definition of $\mdim_\FF(X)$ is in fact a true limit and that $\mdim_\FF(X)$ is independent of the choice of the right F\o lner net $\FF$ for $G$.
These two important facts  can be deduced from the theory of quasi-tiles in amenable groups developed by  
D. Ornstein and B. Weiss in \cite{ornstein-weiss} (see \cite{gromov-tids}, \cite{elek-rank}, \cite{krieger-lemme}). However, we do not need them in the present paper.
\end{remark*}

\subsection{Reversible linear cellular automata}

Let $G$ be a group and let $A$ be a set. 
A  cellular automaton $\tau \colon X \to Y$ between
subshifts $X, Y \subset A^G$  is said to be  \emph{reversible} if $\tau$ is bijective and the inverse map 
$\tau^{-1} \colon Y \to X$ is also a cellular automaton.
\par
It is well known that every bijective linear cellular automaton
$\tau \colon X \to Y$ between subshifts $X,Y \subset A^G$ is reversible
when the alphabet $A$ is finite
(this may be easily deduced from the compactness of $A^G$ and the Curtis-Hedlund theorem 
\cite{hedlund-endomorphisms} which says that, when the alphabet $A$ is finite,  a map between subshifts of $A^G$ is a cellular automaton if and only if it is continuous and $G$-equivariant).
On the other hand, if $G$ contains an element of infinite order and $A$ is infinite then one can construct a bijective  cellular automaton
$\tau \colon A^G \to A^G$ which is not reversible (see \cite[Corollary 1.2]{periodic}).
Similarly, if $G$ contains an element of infinite order and $V$ is an infinite-dimensional vector space then one can construct a bijective linear cellular automaton
$\tau \colon V^G \to V^G$ which is not reversible (see \cite[Theorem 1.1]{periodic}).
\par
 The following result is proved in \cite{israel}:

\begin{theorem}
\label{t;reversible-countable}
Let $G$ be a countable group, $V$ a finite-dimensional
vector space over a field $\K$, and  $X, Y \subset V^G$ two linear subshifts.
Then every  bijective linear cellular automaton $\tau \colon X \to Y$ is reversible.
\end{theorem}

\begin{proof}
See \cite[Theorem 3.1]{israel}.
\end{proof}

We will use the fact that mean dimension of linear subshifts is preserved by reversible linear cellular automata:

 \begin{proposition}
\label{p:inv-mdim-reversible}
Let $G$ be an amenable group, $\FF = (F_j)_{j \in J}$ a right F\o lner net for $G$, and
$V$  a finite dimensional vector space over a field $\K$. 
Let $X, Y \subset V^G$ be two linear subshifts. Suppose that there exists a reversible linear cellular automaton $\tau \colon X \to Y$.
Then one has $\mdim_\FF(X) = \mdim_\FF(Y)$.
\end{proposition}

\begin{proof}
As $\tau \colon X \to Y$ is a surjective linear cellular automaton, we have
$\mdim_\FF(Y)  \leq \mdim_\FF(X)$ by Proposition \ref{p;tau}.
Similarly, we have $\mdim_\FF(X) \leq \mdim_\FF(Y)$ since $\tau^{-1} \colon Y \to X$ is a surjective linear cellular automaton.
Thus we have $\mdim_\FF(X) = \mdim_\FF(Y)$. 
\end{proof}

By combining Theorem \ref{t;reversible-countable} and Proposition \ref{p:inv-mdim-reversible}, we get:

 \begin{corollary}
\label{c:inv-mdim-reversible-countable}
Let $G$ be a countable  amenable group, $\FF = (F_j)_{j \in J}$ a right F\o lner net for $G$, and
$V$  a finite dimensional vector space over a field $\K$. 
Let $X, Y \subset V^G$ be two linear subshifts. 
Suppose that there exists a bijective linear cellular automaton $\tau \colon X \to Y$.
Then one has $\mdim_\FF(X) = \mdim_\FF(Y)$.
\qed
\end{corollary}

\section{Mean dimension of strongly irreducible linear subshifts}
\label{sec:mean-dim-strongly-irr}

This section contains  results on mean dimension of strongly irreducible linear subshifts. We start with a slightly technical lemma which will also be used in the next section:

 \begin{lemma} 
\label{l:mdimY<mdimX-strongly}
Let $G$ be an amenable group, $\FF = (F_j)_{j \in J}$ a right F\o lner net for $G$, and $V$ a 
finite-dimensional vector space over a field $\K$.
Let $X \subset V^G$ be  a strongly irreducible linear subshift
and let $\Delta$ be a finite subset of $G$ such that $1_G \in \Delta$ and $X$ is $\Delta$-irreducible.
Let $D$, $E$ and $F$ be finite subsets of $G$ with $D^{+ \Delta} \subset E$.
Suppose that $T \subset G$ is an $(E,F)$-tiling and 
 that $Z$ is a vector subspace of $X$ such that
 \begin{equation}
\label{e:pi-gD-Y-strict-in-pi-gD-X-lemma}
\pi_{gD}(Z) \subsetneqq \pi_{gD}(X)
\end{equation}
 for all $ g \in T$. 
Then one has $\mdim_\FF(Z) < \mdim_\FF(X)$.
\end{lemma}

\begin{proof}
As in Lemma \ref{l;tiling},
let us define, for each $j \in J$, the subset 
$T_j \subset T$ by $T_j = \{g \in T : gE \subset F_j \}$.
Observe that, for all $j \in J$ and $g \in T_j$, we have the inclusions $gD \subset gD^{+\Delta} \subset gE \subset F_j$.
Denote, for $j \in J$ and $g \in T_j$, by $\pi_{gD}^{F_j} \colon V^{F_j} \to V^{gD}$ the natural projection map.
Consider,  for each $j \in J$, the vector subspace
  $\pi_{F_j}^*(X) \subset \pi_{F_j}(X)$ defined by
$$
\pi_{F_j}^*(X) = \{q \in \pi_{F_j}(X): \pi_{gD}^{F_j}(q) \in \pi_{gD}(Z) \mbox{ for all } g \in T_j\}.  
$$
We claim that
\begin{equation}
\label{xi-T-j}
\dim (\pi_{F_j}^*(X))  \leq \dim (\pi_{F_j}(X)) - |T_j|
\end{equation}
for all $j \in J$.
\par
To prove our claim, let us fix an element $j \in J$ and suppose that  $T_j = \{g_1, g_2, \ldots, g_m\}$, where $m = \vert T_j \vert$. 
Consider, for each $i \in \{0,1,\ldots,m\}$, 
the vector subspace $\pi_{F_j}^{(i)}(X) \subset \pi_{F_j}(X)$
defined by
$$
\pi_{F_j}^{(i)}(X) = \{q \in \pi_{F_j}(X): \pi_{g_kD}^{F_j}(q) \in \pi_{g_kD}(Z)  \mbox{ for all } 1 \leq k \leq i\}.
$$ 
  Note that
  $$
\pi_{F_j}^{(i)}(X) \subset \pi_{F_j}^{(i - 1)}(X)
  $$
 for all $i=1,2,\ldots, m$.
Let us show  that
\begin{equation}
\label{xi-T-j-i}
\dim (\pi_{F_j}^{(i)}(X))  \leq \dim (\pi_{F_j}(X)) - i
\end{equation}
for all $i \in \{0,1,\ldots, m\}$.
Since $\pi_{F_j}^{(m)}(X) = \pi_{F_j}^*(X)$, this will prove \eqref{xi-T-j}.
\par
To establish \eqref{xi-T-j-i}, we use induction on $i$.
For $i = 0$, we have $\pi_{F_j}^{(i)}(X) = \pi_{F_j}(X)$ so that there is nothing to prove.
 Suppose now that $\dim (\pi_{F_j}^{(i-1)}(X))  \leq \dim (\pi_{F_j}(X)) - (i-1)$ for some $i \leq m - 1$.
By hypothesis \eqref{e:pi-gD-Y-strict-in-pi-gD-X-lemma}, we can find an element 
 $p \in \pi_{g_{i}D}(X) \setminus \pi_{g_{i}D}(Z)$.
As $(g_{i}D)^{+\Delta} = g_{i}D^{+\Delta} \subset g_i E$ and $X$ is $\Delta$-irreducible,
there exists an element $x \in X$ such that
$\pi_{g_iD}(x) = p$ and $x$ is identically zero on $F_j \setminus g_iE$.
Now observe that $\pi_{F_j}(x) \in \pi_{F_j}^{(i-1)}(X)$
since the sets $g_1D, g_2D,\ldots, g_{i - 1}D$ are all contained in $F_j \setminus g_i E$.
On the other hand, we have $\pi_{F_j}(x) \notin \pi_{F_j}^{(i)}(X)$ as $\pi_{g_iD}(x) = p \notin \pi_{g_iD}(Z)$.
This shows that $\pi_{F_j}^{(i)}(X)$ is strictly contained in $\pi_{F_j}^{(i-1)}(X)$.  
Hence we have $\dim (\pi_{F_j}^{(i)}(X)) \leq \dim (\pi_{F_j}^{(i-1)}(X)) -1   \leq (\dim (\pi_{F_j}(X)) - (i-1)) - 1 = \dim (\pi_{F_j}(X)) - i$, by using our induction hypothesis. This establishes \eqref{xi-T-j-i} and therefore \eqref{xi-T-j}. 
\par
By Lemma \ref{l;tiling}, we can find a real number $\alpha > 0$ and an element $j_0 \in J$ such that
$\vert T_j \vert \geq \alpha \vert F_j\vert$ for all $j \geq j_0$. 
Since $\pi_{F_j}(Z) \subset \pi_{F_j}^*(X)$, we deduce from \eqref{xi-T-j} that 
$\dim (\pi_{F_j}(Z)) \leq \dim (\pi_{F_j}(X)) - \alpha \vert F_j \vert$ for all $j \geq j_0$, so that
\[
\begin{split}
\mdim_\FF(Z) & = \limsup_{j}\frac{\dim (\pi_{F_j}(Z))}{\vert F_j \vert}\\ 
& \leq \limsup_{j}\frac{\dim (\pi_{F_j}(X))}{\vert F_j \vert} - \alpha \\ 
& = \mdim_\FF(X) - \alpha \\ 
& < \mdim_\FF(X). \qed
\end{split}
\]
\renewcommand{\qedsymbol}{}
\end{proof}

 \begin{proposition} 
\label{p;mdimY<mdimX}
Let $G$ be an amenable group, $\FF = (F_j)_{j \in J}$ a right F\o lner net for $G$, and $V$ a 
finite-dimensional vector space over a field $\K$. Let $X \subset V^G$ be  a strongly irreducible linear subshift
 and $Y \subset V^G$  a linear subshift such that $Y \subsetneqq X$. 
Then one has $\mdim_\FF(Y) < \mdim_\FF(X)$.
\end{proposition}

\begin{proof}
As $Y \subsetneqq X$ and $Y$ is closed in $V^G$ for the prodiscrete topology,  we can find a finite
subset $D \subset G$ such that $\pi_D(Y) \subsetneqq \pi_D(X)$.
By  the $G$-invariance of $X$ and $Y$,  this implies
 \begin{equation}
\label{e:pi-gD-Y-strict-in-pi-gD-X}
\pi_{gD}(Y) \subsetneqq \pi_{gD}(X)
\end{equation}
for all $g \in G$.
\par
Let $\Delta$ be a finite subset of $G$ such that $1_G \in \Delta$ and $X$ is $\Delta$-irreducible, and 
take  $E = D^{+\Delta}$. By virtue of Lemma \ref{l;tilings-exist}, we can find a finite subset $F \subset G$ and an $(E,F)$-tiling  $T \subset G$.
Then, by taking $Z = Y$,  all the hypotheses in Lemma \ref{l:mdimY<mdimX-strongly} are satisfied so that 
we get $\mdim_\FF(Y) < \mdim_\FF(X)$.
\end{proof}
 
\begin{corollary}
\label{cor:positive-mdim-if-strongly}
Let $G$ be an amenable group,  $\FF = (F_j)_{j \in J}$ a right F\o lner net for $G$, 
and $V$  a finite-dimensional vector space over a field $\K$. 
Let $X \subset V^G$ be a nonzero strongly irreducible linear subshift.
 Then one has $\mdim_\FF(X) > 0$. 
 \end{corollary}

\begin{proof}
It suffices to apply Proposition \ref{p;mdimY<mdimX} by taking $Y = \{0\}$.
\end{proof}

Corollary \ref{cor:positive-mdim-if-strongly} becomes false if we suppress the hypothesis that $X$ is strongly irreducible even   for irreducible linear subshifts of finite type as the following example shows.

\begin{example}
Take $G = \Z^2$ and the F\o lner sequence $\FF = (F_n)_{n \geq 1}$ given by 
$F_n = \{0,1,\ldots,n - 1\}^2$ for all $n \geq 1$.
Let $\K$ be a field,  $V$ a nonzero finite-dimensional vector space over $\K$,
and consider the subset $X \subset V^G$ defined by
$$
X = \{x \in V^G : x(g) = x(h) \text{ for all } g,h \in G \text{ such that } \rho(g) = \rho(h) \},
$$
where $\rho \colon \Z^2 = \Z \times \Z \to \Z$ denotes the projection onto the second factor.
In other words, $X$ consists of the configurations which are constant on each horizontal line in $\Z^2$.
Observe that $X$ is a linear subshift of finite type with defining window $D = \{(0,0),(1,0)\}$ and defining law
$L = \{y \in V^D : y(0,0) = y(1,0)\}$.
On the other hand, $X$ is irreducible.
Indeed, this immediately follows from the fact that  if $\Omega$ is a finite subset of $G$, then we can translate $\Omega$ vertically to get a subset $\Omega' \subset \Z^2$ such that $\Omega$ and $\Omega'$ have disjoint images under the projection  $\rho$.
\par
However, we have $\dim(\pi_{F_n}(X)) = n \dim(V)$ and $\vert F_n \vert = n^2$ for all $n \geq 1$  so that
$\mdim_\FF(X) = \lim_{n \to \infty} n^{-1} \dim(V) = 0$.
\end{example}

\section[The Mittag-Leffler lemma]{The Mittag-Leffler lemma and the closed image property for linear subshifts}
  \label{sec:Mittag-Leffler-CIP}

This section contains the proof of Theorem \ref{t:strongly-irr-are-L-surj}.
\par 
Let $G$ be a group and let $A$ be a set. 
\par
Suppose first that $A$ is finite and let $\tau\colon X \to A^G$ be a cellular automaton, where $X \subset A^G$ is a subshift. It immediately follows from the compactness of $X$ and the continuity of $\tau$ that the image $\tau(X)$ is closed in $A^G$ for the prodiscrete topology.
As $\tau$ is $G$-equivariant, we deduce that $\tau(X)$ is a subshift of $A^G$.
\par
If $G$ contains an element of infinite order and $A$ is infinite
then one can construct a cellular automaton $\tau \colon A^G \to A^G$
whose image is not closed in $A^G$ (see \cite[Corollary 1.4]{periodic}).
Similarly, if $G$ contains an element of infinite order and $V$ is an infinite-dimensional vector space then one can construct a linear cellular automaton
$\tau \colon V^G \to V^G$ whose image is not closed in $V^G$ (see \cite[Theorem 1.3]{periodic}).  
\par
On the other hand, if $A = V$ is a finite-dimensional vector space over a field $\K$ and
$\tau \colon V^G \to V^G$ is a linear cellular automaton, then the image of $\tau$ is closed in $V^G$ (see \cite[Lemma 3.1]{garden} for $G$ countable and \cite[Corollary 1.6]{induction} in the general case, see also \cite[Section 4.D]{gromov-esav}). As $\tau$ is $G$-equivariant and $\K$-linear, this implies that $\tau(V^G)$ is a linear subshift of $V^G$.
\par 
In this section we extend this last result to linear cellular automata $\tau \colon X \to V^G$, where $G$ is a countable  group, $V$ is a finite-dimensional vector space, and $X \subset V^G$ is a linear subshift.
The key point in the proof relies in a general well known result, namely the Mittag-Leffler lemma for projective sequences of sets. This version of the Mittag-Leffler lemma  may be easily deduced from Theorem 1 in \cite[TG II. Section 5]{bourbaki-top-gen} (see also \cite[Section I.3]{grothendieck-ega-3}). We give an independent proof here for the convenience of the reader.
Let us first recall a few facts about projective limits of projective sequences in the category of sets.
\par
Let $\N$ denote the set of nonnegative integers.
A \emph{projective sequence} of sets  consists of a sequence $(X_n)_{n \in \N}$ of sets together with maps $f_{nm} \colon X_m \to X_n$ defined for all $ m \geq n$
which  satisfy the following conditions:

\begin{enumerate}[(PS-1)]
\item
$f_{n n}$ is the identity map on $X_n$ for all $n \in \N$;
\item
$f_{n k} = f_{n m} \circ f_{m k}$ for all $n,m,k \in \N$ such that $k \geq m \geq n$.
\end{enumerate}
Such a projective sequence will be denoted $(X_n,f_{n m})$ or simply $(X_n)$.
The \emph{projective limit} $\varprojlim X_n$ of the projective sequence $(X_n,f_{n m})$ is the subset of $\prod_{n \in \N} X_n$ consisting of the sequences $(x_n)_{n \in \N}$ satisfying 
$x_n = f_{n m}(x_m)$ for all $n,m \in \N$ such that $n \leq m$.
\par	
We say that the projective sequence $(X_n)$ satisfies the \emph{Mittag-Leffler condition} if,
for each $n \in \N$, there exists $m \in \N$ such that $f_{n k}(X_k) = f_{n m}(X_m)$ for all $k \geq m$.

\begin{lemma}[Mittag-Leffler]
\label{l;ML}
If  $(X_n,f_{n m})$ is a projective system of nonempty sets
which satisfies the Mittag-Leffler condition
then its projective limit $X = \varprojlim X_n$ is not empty.
\end{lemma}

\begin{proof}
First observe that if $(X_n,f_{n m})$ is an arbitrary  projective sequence of sets,
then Property (PS-2) implies that, for each $n \in \N$, the sequence of sets $f_{nm}(X_m)$, $m \geq n$, is 
non-increasing.
The set $X_n' = \bigcap_{m \geq n} f_{nm}(X_m)$ is called the set of \emph{universal elements} in $X_n$ (cf. \cite{grothendieck-ega-3}).
It is clear that the map $f_{n m}$ induces by restriction a map $g_{n m} \colon X_m' \to X_n'$ for all $n \leq m$ and  that $(X_n',g_{n m})$ is a projective sequence 
having the same projective limit as the projective sequence $(X_n,f_{n m})$.
\par
Suppose now that all the sets $X_n$ are nonempty and that the projective sequence $(X_n,f_{n m})$ satisfies the Mittag-Leffler condition.
This means that, for each $n \in \N$, there is an integer $m \geq n$ such that $f_{n k}(X_k) = f_{n m}(X_m)$ for all $k \geq m$.
This implies $X_n' = f_{n m}(X_m)$ so that, in particular, the set $X_n'$ is not empty.
We claim that the map $g_{n, n + 1} \colon X_{n + 1}' \to X_n'$ is surjective for every $n \in \N$. To see this, let $n \in \N$ and $x_n' \in X_n'$. By the Mittag-Leffler condition, we can find an integer $p \geq n + 1$ such that
$f_{n k}(X_k) = f_{n p}(X_p)$ and $f_{n + 1,  k}(X_k) = f_{n + 1,p}(X_p)$ for all $k \geq p$. 
It follows that $X_n' = f_{n p}(X_p)$ and $X_{n + 1}' = f_{n + 1, p}(X_p)$. Consequently, we can find $x_p \in X_p$ such that $x_n' = f_{n p}(x_p)$.
Setting $x_{n + 1}' = f_{n + 1,p}(x_p)$, we have  $x_{n + 1}' \in X_{n + 1}'$ and 
$$
g_{n, n + 1}(x_{n +1}') = f_{n, n + 1}(x_{n +1}') = f_{n, n + 1} \circ f_{n+1,p} (x_p) = f_{n p}(x_p) = x_n'.
$$ 
This proves our claim that $g_{n, n + 1}$ is onto.
Now, as the sets $X_n'$ are nonempty,   we can construct by induction a sequence $(x_n')_{n \in \N}$ such that $x_n' = g_{n, n + 1}(x_{n +1}') $ for all $n \in \N$.
This sequence is in the projective limit $\varprojlim X_n' = \varprojlim X_n$. This shows that $\varprojlim X_n$ is not empty.    
\end{proof} 

\begin{theorem}
\label{t;closure-lin}
Let $G$ be a countable group and let $V$ be a finite-dimensional vector space over a field $\K$.
Let $\tau \colon X \to V^G$ be a linear cellular automaton, where $X \subset V^G$ is a linear subshift.  
Then $\tau(X)$ is closed in $V^G$ for the prodiscrete topology and
is therefore a linear subshift of $V^G$.
\end{theorem}

\begin{proof}
Since $G$ is countable, we can find a sequence $(A_n)_{n \in \N}$ of finite subsets of $G$ such
that $G = \bigcup_{n \in \N} A_n$ and $A_n \subset A_{n + 1}$ for all $n \in \N$. Let $M$ be a memory
set for $\tau$. Let $B_n = \{g \in G: gM \subset A_n\}$. Note that $G = \bigcup_{n \in \N}B_n$ and $B_n \subset B_{n + 1}$ for all $n \in \N$. Denote by $\pi_{A_n} \colon V^G \to V^{A_n}$ and $\pi_{B_n} \colon V^G \to V^{B_n}$, $n \in \N$, the corresponding projection  maps.
\par  
Since $M$ is a memory set for $\tau$, it follows from \eqref{e;local-property} that if $x$ and $x'$ are elements in $X$ such that $\pi_{A_n}(x) = \pi_{A_n}(x')$ then $\pi_{B_n}(\tau(x)) = \pi_{B_n}(\tau(x'))$. Therefore, given
$x_n \in \pi_{A_n}(X)$ and denoting by $\widetilde{x}_n$ any configuration in $X$ such that 
$\pi_{A_n}(\widetilde{x}_n) = x_n$,
the element
$$
y_n = \pi_{B_n}(\tau(\widetilde{x}_n)) \in V^{B_n}
$$ 
does not depend on the particular choice of the extension $\widetilde{x}_n$. Thus we can define a map 
$\tau_n \colon \pi_{A_n}(X) \to V^{B_n}$ by setting $\tau_n(x_n) = y_n$ for all $x_n \in \pi_{A_n}(X)$. 
It is clear that $\tau_n$ is $\K$-linear.
\par
Let now $y \in V^G$ and suppose that $y$ is in the closure of $\tau(X)$.
Then, for all $n \in \N$, there exists $z_n \in X$ such that 
\begin{equation}
\label{e;y-z-n-B-n}
\pi_{B_n}(y)=  \pi_{B_n}(\tau(z_n)).
\end{equation}
Consider, for each $n \in \N$, the affine
subspace $X_n \subset \pi_{A_n}(X)$ defined by $X_n = \tau_n^{-1}(\pi_{B_n}(y))$.
We have $X_n \not= \varnothing$ for all $n$ by \eqref{e;y-z-n-B-n}. 
For $n \leq m$, the restriction map $\pi_{A_m}(X) \to \pi_{A_n}(X)$ induces an affine map
$f_{nm} \colon X_m \to X_n$. Conditions (PS-1) and (PS-2) are trivially satisfied so that
$(X_n,f_{nm})$ is a projective sequence. We claim that $(X_n,f_{nm})$
satisfies the Mittag-Leffler condition. Indeed, consider, for all $n \leq m$, the
affine subspace $f_{nm}(X_m) \subset X_n$. We have $f_{nm'}(X_{m'})  \subset f_{nm}(X_m) $ for all
$n \leq m \leq m'$ since $f_{nm'} = f_{nm} \circ f_{mm'}$.
As the sequence  $f_{nm}(X_m)$ ($m = n,n+1,\dots$) is a
non-increasing sequence of finite-dimensional affine subspaces, it
stabilizes, i.e., for each $n \in \N$ there exists an integer $m \geq n$ such that
$f_{nk}(X_k) = f_{nm}(X_m)$ if $k \geq m$. Thus, the Mittag-Leffler condition is satisfied.
It follows from Lemma \ref{l;ML} that the projective limit $\varprojlim X_n$ is nonempty.
Choose an element $(x_n)_{n \in \N} \in \varprojlim X_n$. We have that
$x_{n + 1}$ coincides with $x_n$ on $A_n$ and that $x_n \in \pi_{A_n}(X)$ for all $n \in \N$. 
As $X$ is closed in $V^G$  and $G = \cup_{n \in \N} A_n$, we deduce that there exists a (unique) configuration 
$x \in X$ such that $x\vert_{A_n} = x_n$ for all $n$. 
We have $\tau(x)\vert_{B_n}= \tau_n(x_n) = y_n = y\vert_{B_n}$ for all $n$. 
Since $G = \cup_{n \in \N} B_n$, this shows that $\tau(x) = y$.
\end{proof}

\begin{corollary}
\label{c:mdim-strongly-countable-surj}
Let $G$ be a countable amenable group, $\FF = (F_j)_{j \in J}$ a right F\o lner net for $G$, and $V$ a 
finite-dimensional vector space over a field $\K$. 
Let $\tau \colon X \to Y$ be a linear cellular automaton, 
where $X,Y \subset V^G$ are linear subshifts such that $\mdim_\FF(X) = \mdim_\FF(Y)$ and
$Y$ is strongly irreducible.
Then the following conditions are equivalent:
\begin{enumerate}[\rm (a)]
\item 
$\tau$ is surjective;
\item 
$\mdim_\FF(\tau(X)) = \mdim_\FF(X)$.
\end{enumerate}
\end{corollary}

\begin{proof} 
The implication (a) $\Rightarrow$ (b) is trivial. Conversely, suppose that $\mdim_\FF(\tau(X)) = \mdim_\FF(X)$. 
 Theorem \ref{t;closure-lin} implies that  $\tau(X)$ is a linear subshift of $V^G$. As $\tau(X) \subset Y$,
it then follows from Proposition \ref{p;mdimY<mdimX} that $\tau(X) = Y$. Thus, $\tau$ is surjective.
\end{proof}

\begin{proof}[Proof of Theorem \ref{t:strongly-irr-are-L-surj}]
Let $X \subset V^G$ be a strongly irreducible linear subshift and suppose that $\tau \colon X \to X$ is an injective linear cellular automaton. Let us show that $\tau$ is surjective.
Let $\FF = (F_j)_{j \in J}$ be a right F\o lner net for $G$.
We know that  $\tau(X)$ is a linear subshift by Theorem \ref{t;closure-lin}.
 As $\tau$ induces a bijective linear cellular automaton from $X$ onto $\tau(X)$,
 we have $\mdim_\FF(\tau(X)) = \mdim_\FF(X)$ by using Corollary \ref{c:inv-mdim-reversible-countable}. 
 Since $X$ is strongly irreducible,
this implies that $\tau$ is surjective by Corollary \ref{c:mdim-strongly-countable-surj}.
Thus $X$ is $L$-surjunctive. 
 \end{proof}

\section{The closed image property for linear subshifts of finite type}
\label{sec:cip-sflt}

In this section we show that Theorem \ref{t;closure-lin}
remains true for any (possibly uncountable) group $G$ if we add the hypothesis that  the linear subshift $X \subset V^G$ is of finite type.
The proof relies on the fact that  a subshift of finite type can be factorized along the left cosets of any subgroup containing a defining window. 
In order to state this last result in a more precise way, 
let us first introduce some notation.
\par
Let $G$ be a group and let $A$ be a set.
Let $H$  be a subgroup of $G$ and denote by $G/H = \{gH : g \in G\}$ the set consisting of all left cosets of $H$ in $G$. For every coset $c \in G/H$, we equip the set $A^c = \prod_{g \in c} A$ with its prodiscrete topology and
we denote by $\pi_c \colon A^G \to A^c$ the projection  map. Since the cosets $c \in G/H$ form a partition of $G$, we have a natural identification of topological spaces
$$
A^G = \prod_{c \in G/H} A^c.
$$
 With this identification, we have 
$x= (x\vert_c)_{c \in G/H}$
 for each $x \in A^G$, where $x\vert_c = \pi_c(x)  \in A^c$ is the restriction of the configuration $x$ 
 to  $c$.
\par
Given a coset $c \in G/H$ and an element $g \in c$, let $\phi_g \colon H \to c$ denote the bijective map defined by $\phi_g(h) = gh$ for all $h \in H$. 
Then $\phi_g$ induces a homeomorphism  
$\phi_g^* \colon A^c \to A^H$ 
given by $\phi_g^*(y) = y \circ \phi_g$ for all $y \in A^c$.

\begin{proposition}
\label{p;factorization-subshift}
Let $G$ be a group and let $A$ be a set.
Let $X \subset A^G$ be a subshift of finite type.
Let $D \subset G$ be a defining window and  $L \subset A^D$ a defining law for $X$, so that
$X = X_G(D,L)$. 
Suppose that  $H$  is a subgroup of $G$  such that $D \subset H$.
 Then  one has 
\begin{enumerate}[\rm (i)]
 \item
 $X = \prod_{c \in G/H} X_c$, where $X_c = \pi_c(X) \subset A^c$ denotes the projection of $X$ on 
 $A^c$; 
 \item
 $X_H = X_H(D,L)$;
 \item 
$\phi_g^*(X_c)  = X_H$
for all $c \in G/H$ and $g \in c$.
\end{enumerate}
\end{proposition}

\begin{proof}
 In order to establish (i), it suffices to show that $\prod_{c \in G/H} X_c \subset X$ since the converse inclusion is trivial.
Suppose that  $\widetilde{x} = (\widetilde{x}\vert_c)_{c \in G/H} \in \prod_{c \in G/H} X_c$. 
Let $g \in G$ and consider the left coset $c = gH$ . Then we can find $x \in X$ such that 
$ \widetilde{x}\vert_c = x\vert_c$. 
As $gD \subset gH = c$, we have
$$
(g^{-1}\widetilde{x})(d) = \widetilde{x}(gd) =   x(gd) = (g^{-1}x)(d)
$$
for all $d \in D$.
It follows that $(g^{-1}\widetilde{x})\vert_{D} = (g^{-1}x)\vert_{D} \in L$ for all $g \in G$.
We deduce that $\widetilde{x} \in X_G(D,L) = X$. This completes the proof of (i).
\par
If $x \in X$ then $(h^{-1}x\vert_H)\vert_D = (h^{-1}x)\vert_D \in L$ for all $h \in H$. Thus, we have $X_H \subset X_H(D,L)$.
Conversely, suppose that $y \in X_H(D,L)$.
Choose a complete set of representatives $R \subset G$ for the left cosets of $H$ in $G$ and consider the configuration $x \in A^G$ defined by
$x(rh) = y(h)$ for all $r \in R$ and $h \in H$. Then we clearly have $x \in X_G(D,L) = X$ and $x\vert_H = y$. Thus $X_H(D,L) \subset X_H$. This completes the proof of (ii).
\par
Let now $c \in G/H$ and $g \in c$. In order to prove
 \begin{equation}
\label{e;X-prime-contenu-X-H}
\phi_g^*(X_c)  \subset X_H,
\end{equation} 
let $y_c \in \phi_g^*(X_c)$. Then there exists a (unique) $x_c \in X_c$ such that $y_c = \phi_g^*(x_c)$. Let $x \in X$ such that $\pi_c(x) = x_c$.  
For all $h \in H$ and $d \in D$, we have
$$
(h^{-1}y_c)(d) = y_c(hd) = x_c(ghd) = x(ghd)  = (gh)^{-1}x(d),
$$
so that, $(h^{-1}y_c)\vert_{D} = ((gh)^{-1}x)\vert_{D} \in L$ since $x \in X = X_G(D,L)$. 
This shows
that $y_c \in X_H(D,L) = X_H$ and \eqref{e;X-prime-contenu-X-H} follows.
Conversely, suppose that $x_H \in X_H$ and consider the configuration
$x_c  = \phi_g^*(x_H) \in A^c$. Let us show that $x_c \in X_c$.
Since $x_H \in X_H$, we can find a configuration $x \in X$ such that $x_H = x\vert_H$. Setting $y = gx \in X$,
we have 
$$
  y(gh) = g^{-1}y(h) = x(h)   = x_H(h) = x_c(gh),
$$
for all $h \in H$. Thus $x_c = y\vert_c \in X_c$. This gives $X_H \subset \phi_g^*(X_c)$. 
From this and \eqref{e;X-prime-contenu-X-H} we finally deduce (iii).
\end{proof}

\begin{corollary}
\label{c:sft-is-infinite}
Suppose that  $G$ is a group which is not finitely generated.
Then:
\begin{enumerate}[\rm (i)]
\item
if $A$ is a set and $X \subset A^G$ is a subshift of finite type which is not reduced to a single configuration then $X$ is infinite;
\item
if $V$ is a vector space over a field $\K$ and $X \subset V^G$ is a linear subshift of finite type which is not reduced to the zero configuration then $X$ is 
infinite-dimensional (as a vector space over $\K$). 
\end{enumerate}
\end{corollary}

\begin{proof}
Let $A$ be a set, $X \subset A^G$ a subshift of finite type, and $D \subset G$ a defining window for $X$.
Let $H$ denote the subgroup of $G$ generated by $D$.
Observe that $H$ is of infinite index in $G$ since $G$ is not finitely generated.
With the above notation, we have $X = \prod_{c \in G/H} X_c$ by Proposition \ref{p;factorization-subshift}.
Moreover, for all $c \in G/H$ and $g \in c$, we have $\phi_g^*(X_c)  = X_H$.
As all the maps $\phi_g^*$ are bijective, we deduce that $X$ is either reduced to a single configuration or infinite. This proves (i).
\par
Suppose now that $A = V$ is a vector space over some field $\K$.
Then $X_c$ is a vector subspace of $V^c$ and
$\phi_g^* \colon X_c \to X_H$ is an isomorphism of $\K$-vector spaces for all $c \in G/H$ and $g \in c$.
As $X = \prod_{c \in G/H} X_c$, we conclude that $X$ is either reduced to the zero configuration of infinite-dimensional.
This shows (ii).
\end{proof}

\begin{example}
Let $G$ be a group which is not finitely generated and let $V$ be a nonzero finite-dimensional vector space over a field $\K$. Suppose that $G_0$ is a finite index subgroup of $G$.
Consider the linear subshift  $X \subset V^G$  consisting of the configurations $x \in V^G$ which are fixed by each element of $G_0$. We clearly have $\dim(X) = [G:G_0]\dim(V) < \infty$. Thus $X$ is not of finite type by Corollary \ref{c:sft-is-infinite}.(ii).
\end{example}

\begin{theorem}
\label{t;closure-lin-2}
Let $G$ be a (possibly uncountable) group and let $V$ be a finite-dimensional vector space over a field $\K$.
Let $\tau \colon X \to V^G$ be a linear cellular automaton, where $X \subset V^G$ is a linear subshift of finite type. 
Then $\tau(X)$ is closed in $V^G$ for the prodiscrete topology
and is therefore a linear subshift of $V^G$.
\end{theorem}

\begin{proof}
Let $M \subset G$ be a memory set and $\mu \colon V^{M} \to V$ a local defining map for $\tau$. 
Also let $D \subset G$ be a defining window for $X$ and denote by $H$ the subgroup of $G$ generated by $M$ and $D$. 
Note that $H$ is finitely generated since both $M$ and $D$ are finite sets.
\par
Setting $X_c = \pi_c(X)$ for al $c \in G/H$, we have
$X = \prod_{c \in G/H} X_c$ by Proposition \ref{p;factorization-subshift}.
On the other hand, if $x \in X$, $c \in G/H$, and $g \in c$
then $\tau(x)(g)$ depends only on the restriction of  $x$ to $c$,  since $gM \subset gH = c$.
This implies that $\tau$ may be written as a product
\begin{equation}
\label{e;tau-prod}
\tau = \prod_{c \in G/H} \tau_c,
\end{equation}
where $\tau_c \colon X_c    \to V^c$ is the unique map which satisfies 
$\tau_c(x\vert_c) =  (\tau(x))\vert_c$ for all $x \in X$. 
Note that $\tau_H \colon X_H \to V^H$ is the linear cellular automaton over $H$ with memory 
set $M \subset H$ and local defining map $\mu$. 
\par  
Let us show that the maps $\tau_c$ and $\tau_H$ are conjugate by $\phi_g^*$, that is, 
\begin{equation}
\label{e;tau-c-tau-H}
\tau_c = (\phi_g^*)^{-1} \circ \tau_H \circ \phi_g^*.
\end{equation}
Let $y \in X_c$ and let $x \in X$ extending $x$.
For all $h \in H$, we have
\begin{align*}
(\phi_g^* \circ \tau_c)(y)(h) 
&= \phi_g^*(\tau_c(y))(h) \\
&= (\tau_c(y) \circ \phi_g)(h) \\
&= \tau_c(y)(gh) \\
&= \tau(x)(gh) \\
&= g^{-1}\tau(x)(h) \\
&= \tau(g^{-1}x)(h),
\end{align*}
where the last equality follows from the $G$-equivariance of $\tau$.
Now observe that the configuration $g^{-1}\widetilde{x} \in X$ extends $x \circ \phi_g \in X_H$. Thus, 
we have
$$
(\phi_g^* \circ \tau_c)(x)(h) = \tau_H(x \circ \phi_g)(h) = \tau_H(\phi_g^*(x))(h) = (\tau_H \circ \phi_g^*)(x)(h).
$$
This shows that $\phi_g^* \circ \tau_c = \tau_H \circ \phi_g^*$, which gives \eqref{e;tau-c-tau-H}
since $\phi_g^*$ is bijective.
\par
As the subgroup $H \subset G$ is finitely generated and therefore countable, we deduce from Theorem  \ref{t;closure-lin} that
$\tau_H(X_H)$ is closed in $V^H$ for the prodiscrete topology. Since $\phi_g^*$ is a homeomorphism,
it follows that 
$$
\tau_c(X_c) =  (\phi_g^*)^{-1} (\tau_H(X_H))
$$
is closed in $V^c$ for all $c \in G/H$. Thus,
$$
\tau(X) = \prod_{c \in G/H} \pi_c(\tau(X)) = \prod_{c \in G/H} \tau_c(X_c)
$$
is a closed subspace of $V^G$.
\end{proof}

\begin{corollary}
\label{c:surj-X-to-Y-strongly-irred-equiv-mdim}
Let $G$ be a (possibly uncountable) amenable group, $\FF = (F_j)_{j \in J}$ a right F\o lner net for $G$, and $V$ a 
finite-dimensional vector space over a field $\K$. 
Let $\tau \colon X \to Y$ be a linear cellular automaton, 
where $X,Y \subset V^G$ are linear subshifts satisfying $\mdim_\FF(X) = \mdim_\FF(Y)$. Suppose that
$X$ is of finite type and that $Y$ is strongly irreducible.
Then the following conditions are equivalent:
\begin{enumerate}[\rm (a)]
\item 
$\tau$ is surjective;
\item 
$\mdim_\FF(\tau(X)) = \mdim_\FF(X)$.
\end{enumerate}
\end{corollary}

\begin{proof} 
The implication (a) $\Rightarrow$ (b) is trivial. Conversely, suppose that 
$\mdim_\FF(\tau(X)) = \mdim_\FF(X)$.  Theorem \ref{t;closure-lin-2} 
  implies that $\tau(X)$ is a linear subshift of $V^G$. As $\tau(X) \subset Y$,
it then follows from Proposition \ref{p;mdimY<mdimX} that $\tau(X) = Y$. Thus, $\tau$ is surjective.
\end{proof}

 \section{Proof of the Garden of Eden theorem}
\label{sec:proof-GOE}

This section contains the proof of Theorem \ref{t;main}.
Let us start by the following:
 
\begin{theorem}
\label{t;trois}
Let $G$ be an amenable group, $\FF = (F_j)_{j \in J}$ a right F\o lner net for $G$,  and $V$ a 
finite-dimensional vector space over a field $\K$.
Let $X \subset V^G$ be a strongly irreducible linear subshift of finite type
and let $\tau \colon X \to V^G$ be a linear cellular automaton. 
Then the following conditions are equivalent:
\begin{enumerate}[\rm (a)]
\item 
$\tau$ is pre-injective;
 \item 
$\mdim_\FF(\tau(X)) = \mdim_\FF(X)$. 
\end{enumerate}
\end{theorem}

 For the proof of (a) $\Rightarrow$ (b) in Theorem \ref{t;trois}, we shall use the following:  

\begin{lemma}
\label{lemma:strongly-irred-Omega2-inf}
Let $G$ be a group and let $V$ be a finite-dimensional vector space over a field $\K$.
Let $X \subset V^G$ be a strongly irreducible linear subshift of finite type and
suppose that  $M$ is a finite subset of $G$ such that $X$ is $M$-irreducible, $1_G \in M$, and $M^{-1}$ is a defining window for $X$. 
Then, given any configuration $x \in X$ and any finite subset $\Omega \subset G$,
there exists a configuration $z \in X$ which coincides with $x$ on $\Omega$ and is identically zero 
on $G \setminus \Omega^{+M}$.
\end{lemma}

\begin{proof}
 Let $x \in X$ and $\Omega \subset G$ a finite subset.
Note that we have the inclusions  $\Omega \subset \Omega^{+M} \subset \Omega^{+M^2} \subset \Omega^{+M^3}$ since $1_G \in M$.
  As both $x$ and the zero configuration belong to $X$ and $X$ is $M$-irreducible, we can find a configuration $z' \in X$ which coincides with $x$ on 
  $\Omega$  and is identically zero  on $\Omega^{+M^3} \setminus \Omega^{+M}$.
Now consider the configuration $z \in V^G$ which coincides with $z'$ on $\Omega^{+M^3}$ and
is identically zero on $G \setminus \Omega^{+M^3}$.
Observe that if $g \in \Omega^{+M^2}$ then $gM^{-1} \subset \Omega^{+M^3}$ and therefore $z$ coincides with $z'$ on $gM^{-1}$, while if 
$g \in G \setminus \Omega^{+M^2}$ then $gM^{-1} \subset G \setminus \Omega^{+M}$ and therefore  
$z$ is identically zero on $gM^{-1}$.
As both $z'$ and the zero configuration belong to $X$  and  $M^{-1}$ is a defining window for $X$, we deduce that $z \in X$.
On the other hand, $z$ coincides with  $x$ on $\Omega$ and is identically zero on 
$G \setminus \Omega^{+ M}$. 
Consequently, $z$ has the required properties.
 \end{proof}

\begin{proof}[Proof of (a) $\Rightarrow$ (b) in Theorem \ref{t;trois}]
Suppose that $\mdim_\FF(\tau(X)) < \mdim_\FF(X)$. 
Let $Y=\tau(X)$.
Let $M \subset G$ be a memory set for $\tau$. Up to enlarging the subset $M$ if necessary, we can also suppose that $1_G \in M$ and that $X$ is $M$-irreducible and admits $M^{-1}$ as a defining window. 
\par
We first observe that $\pi_{F_j^{+M^2}}(Y)$ is a vector subspace of $\pi_{F_j}(Y) \times V^{F_j^{+M^2} \setminus F_j}$ so that we have
\begin{equation}
\label{e:maj-dim-pi-Fj-M2}
\dim(\pi_{F_j^{+ M^2}}(Y)) \leq \dim (\pi_{F_j}(Y)) + \vert F_j^{+M^2} \setminus F_j \vert \dim(V).
\end{equation}
On the other hand, as $(F_j)_{j \in J}$ is a right F\o lner net for $G$, we have 
 $$
\lim_j \frac{\vert F_j^{+M^2} \setminus F_j \vert}{\vert F_j \vert} = 0
$$
by \eqref{e:folner-net}. Therefore, after dividing the two sides of \eqref{e:maj-dim-pi-Fj-M2} 
by $\vert F_j \vert$ and taking the $\limsup$ over $j$, we get
$$
\limsup_j \frac{\dim(\pi_{F_j^{+ M^2}}(Y))}{\vert F_j \vert} 
\leq \limsup_j \frac{\dim(\pi_{F_j }(Y))}{\vert F_j \vert} = \mdim_\FF(Y).
$$
As $\mdim_\FF(Y) < \mdim_\FF(X)$ by our assumption, this implies that there exists $j_0 \in J$ such that 
 \begin{equation}
\label{e;U-3-leq-X}
\dim(\pi_{F_{j_0}^{+ M^2}}(Y))  < \dim(\pi_{F_{j_0}}(X)).
\end{equation}
\par
Consider now the finite-dimensional vector subspace  $Z \subset X$ consisting of all configurations 
$z \in X$ whose support 
$\{g \in G : z(g) \not= 0 \}$ is contained in $F_{j_0}^{+M}$.
 By virtue of Lemma \ref{lemma:strongly-irred-Omega2-inf}, we have
\begin{equation}
\label{e;Z-leq-X}
\pi_{F_{j_0}}(Z) = \pi_{F_{j_0}}(X).
\end{equation}
On the other hand, we deduce from
Proposition \ref{p:tau-coinc-sur-int}  that $\tau(z)$ is identically zero on $G \setminus F_{j_0}^{+M^2}$ for every $z \in Z$.
Consequently, we have 
\[
\begin{split}
\dim(\tau(Z)) & = \dim(\pi_{F_{j_0}^{+M^2}}(\tau(Z)))\\
&  \leq \dim (\pi_{F_{j_0}^{+M^2}}(Y)) \\
& < \dim(\pi_{F_{j_0}}(X)) \quad (\mbox{by }\eqref{e;U-3-leq-X})\\
& = \dim(\pi_{F_{j_0}}(Z)) \quad (\mbox{by } \eqref{e;Z-leq-X}). \\ 
 \end{split}
\]
As $\dim(\pi_{F_{j_0}}(Z)) \leq \dim(Z)$, this implies $\dim(\tau(Z)) < \dim(Z)$.
It follows that we can find two distinct configurations $z_1,z_2 \in Z$ such that $\tau(z_1) =  \tau(z_2)$.
Since all configurations in $Z$ coincide outside $F_{j_0}^{+M}$,
this shows that $\tau$ is not pre-injective.
\end{proof}

For the proof of (b) $\Rightarrow$ (a), we shall use the following:

\begin{lemma}
\label{l;tilings-disjoints}
Let $G$ be a group and let $V$ be a finite-dimensional vector space over a field $\K$.
Let $X \subset V^G$ be a linear subshift of finite type 
and let $D$ be a defining window for $X$ with $1_G \in D$. 
Let $(\Omega_i)_{i \in I}$ be
a family of subsets of $G$ such that $\Omega_i^{+D} \cap \Omega_j^{+D} = \varnothing$ for all distinct $i,j \in I$. Also let $(x_i)_{i \in I}$ be a family of configurations in $X$ such that the support of $x_i$ is contained in $\Omega_i$  for each $i \in I$. Then the configuration $x \in V^G$ defined by $x(g) = x_i(g)$ if $g \in \Omega_i$ for some (necessarily unique) $i \in I$ and $x(g) = 0$ otherwise, 
satisfies $x \in X$.
\end{lemma}

\begin{proof}
If $g \in \Omega_i^{+D}$ for some (necessarily unique) $i \in I$ then $x$ coincides with $x_i$ on $gD$.
Otherwise, $x$ is identically zero on $gD$. 
As  $D$ is a defining window for $X$, this shows that $x \in X$. 
 \end{proof}

\begin{proof}[Proof of (b) $\Rightarrow$ (a) in Theorem \ref{t;trois}]
Suppose that    $\tau$ is not pre-injective.
This means that  we can find a configuration
$x_0 \in X$ with  finite support $\Omega = \{g \in G : x_0(g) \not= 0 \} \not= \varnothing$  satisfying 
$\tau(x_0) = 0$.
Let $M$ be a memory set for $\tau$. We can also assume that  $1_G \in M$, that $M = M^{-1}$,  and that $M$ is a defining window for $X$. Let $E = \Omega^{+M^2}$. Then, by Lemma \ref{l;tiling}, we can find a finite subset $F \subset G$ and
an $(E,F)$-tiling $T \subset G$. 
Note that, for each $g \in G$, the
support of the configuration $gx_0$ is the set $g\Omega$.
As $g\Omega \subset g\Omega^{+M}$, this implies $\pi_{g\Omega^{+M}}(gx_0) \not= 0$.
 Let us choose, for
each $g \in T$, a hyperplane $H_g \subset \pi_{g \Omega^{+M}}(X)$
such that $\pi_{g\Omega^{+M}}(gx_0) \notin H_g$.
\par  
Consider now the vector subspace $Y  \subset X$ consisting
of all the configurations $y \in X$ which satisfy $\pi_{g\Omega^{+M}}(y) \in H_g$ for all $g \in T$.
We claim that $\tau(Y) = \tau(X)$. To see this, let $x$ be an arbitrary configuration in $X$.
Then, for each $g \in T$, there exists a scalar $\lambda_g \in \K$ such that
$\pi_{g\Omega^{+M}}(x + \lambda_g gx_0) \in H_g$. 
Now observe that $(g\Omega)^{+M} \cap (g'\Omega)^{+M} \subset gE \cap g'E = \varnothing$ for all distinct  $g,g' \in T$  (cf. the defining property (T-1) of a tiling in Section \ref{ss;tilings}). 
Since $X$ is of finite type with defining window $M$ and $1_G \in M$, it follows from Lemma 
\ref{l;tilings-disjoints} that we can find a configuration $x_0' \in X$ such that $\pi_{g\Omega}(x_0') =  \pi_{g\Omega}(\lambda_g gx_0)$ 
for all $g \in T$ and $x_0'$ is identically zero outside $\coprod_{g \in T} g\Omega$.
Note that in fact we have
\begin{equation}
\label{e:restriction}
\pi_{g\Omega^{+M^2}}(x_0') =  \pi_{g\Omega^{+M^2}}(\lambda_g gx_0)
\end{equation} 
for each $g \in T$, since the configuration $gx_0$ is identically zero outside $g\Omega$. 
\par
Consider the configuration $y = x + x_0'$. By construction we have $y \in Y$. Let us show that $\tau(y) = \tau(x)$.
Since $y = x$ outside $\coprod_{g \in T} g\Omega$, we deduce from Proposition \ref{p;tau} that
$\tau(y)$ and $\tau(x)$ coincide outside $\coprod_{g \in T} g\Omega^{+M}$. 
Now, if $h \in g\Omega^{+M}$ for some (necessarily unique) $g \in T$, then 
$hM = hM^{-1} \subset g\Omega^{+M^2}$ and therefore
\begin{align*} 
\tau(y)(h) &= \tau(x + x_0')(h) \\
&= \tau(x + \lambda_g gx_0)(h) \quad \text{(by \eqref{e:restriction})} \\
&= \tau(x)(h) + \lambda_g g\tau(x_0) \quad \text{(by linearity and $G$-equivariance of $\tau$)} \\
&= \tau(x)(h) \quad \text{(since $x_0$ is in the kernel of $\tau$).}
\end{align*}
Thus $\tau(x) = \tau(y)$. This proves our claim that $\tau(X) = \tau(Y)$. 
\par
Using Proposition \ref{p;tau}, we deduce that
\begin{equation}
\label{e:mdim-tauX-mdimY}
\mdim_\FF(\tau(X)) = \mdim_\FF(\tau(Y)) \leq \mdim_\FF(Y). 
\end{equation}
Now observe that, for all $g \in T$, we have $(gx_0)\vert_{\Omega^{+M}} \in
\pi_{g \Omega^{+M}}(X) \setminus \pi_{g \Omega^{+M}}(Y)$ and hence 
 \begin{equation*}
\pi_{g \Omega^{+M}}(Y) \subsetneqq \pi_{g \Omega^{+M}}(X).
\end{equation*}
Therefore, we can apply Lemma \ref{l:mdimY<mdimX-strongly}
to the strongly irreducible linear subshift $X$ and the vector subspace $Y \subset X$ by taking 
$\Delta = M$ and $D = \Omega^{+M}$.
This gives us $\mdim_\FF(Y) < \mdim_\FF(X)$ which, combined with \eqref{e:mdim-tauX-mdimY}, implies
$\mdim_\FF(\tau(X)) < \mdim_\FF(X)$.   
  \end{proof} 

This completes the proof of Theorem \ref{t;trois}.

\begin{corollary}
\label{c;deux}
Let $G$ be an amenable group, $\FF = (F_j)_{j \in J}$ a right F\o lner net for $G$,  and $V$ a 
finite-dimensional vector space over a field $\K$. 
Let $\tau \colon X \to Y$ be a linear cellular automaton, 
where $X,Y \subset V^G$ are linear subshifts satisfying $\mdim_\FF(X) = \mdim_\FF(Y)$. Suppose that $X$ is strongly irreducible of finite type and that
$Y$ is strongly irreducible.
Then the following conditions are equivalent:
\begin{enumerate}[\rm (a)]
\item 
$\tau$ is surjective;
\item 
$\mdim_\FF(\tau(X)) = \mdim_\FF(X)$; 
\item 
$\tau$ is pre-injective.
\end{enumerate}
\end{corollary}

\begin{proof}
The equivalence of conditions (a) and (b) follows from 
Corollary \ref{c:surj-X-to-Y-strongly-irred-equiv-mdim}. 
The equivalence between conditions (b) and (c) follows from Theorem \ref{t;trois}.
\end{proof}

\begin{proof}[Proof of Theorem \ref{t;main}]
This follows immediately from the equivalence between conditions (a) and (c) in Corollary \ref{c;deux} by taking $X = Y$.
\end{proof}





\section{Pre-injective but not surjective linear cellular automata}
\label{sec:pre-inj-but-not-surj}

  In this section we give examples of pre-injective but not surjective linear cellular automata 
$\tau \colon X \to X$, where 
 $G$ is a group, $V$ is a finite-dimensional vector space, and $X \subset V^G$ is a linear subshift.
 We  recall that Theorem \ref{t;main} implies that there is no such example with $X$ strongly irreducible of finite type,
and in particular with $X = V^G$,
if the group $G$ is amenable.
When $G$ contains a nonabelian free subgroup and $\dim(V) = 2$,
one can construct a linear cellular automaton $\tau \colon V^G \to V^G$ which is 
 pre-injective but not surjective. This was done in \cite[Example 4.10]{semi-simple} for free groups of rank $2$  in a more general setting, namely for linear cellular automata whose alphabet is a module over any nonzero  ring. 
 
 \begin{proposition}
\label{p:lca-pre-inj-not-surj}
Let $G$ be a group and let $V$ be a $2$-dimensional vector space over a field $\K$.
Suppose that $G$ contains a nonabelian free subgroup (e.g. $G$ is a nonabelian free group). 
 Then there exists a linear cellular automaton $\tau\colon V^G \to V^G$ which is 
 pre-injective but not surjective.
\end{proposition}

\begin{proof}
We may assume $V = \K^2$.
Let $p_1$ and $p_2$ be the endomorphisms of $V$ defined respectively by
$p_1(v) = (\lambda_1,0)$ and $p_2(v) = (\lambda_2,0)$ for all $v = (\lambda_1,\lambda_2) \in V$.
Let $a$ and $b$ be two elements in $G$  generating a free subgroup of rank $2$. 
Consider the map  $\tau \colon V^G \to V^G$ given by
\begin{equation*}
 \label{e:def-tau-surj-non-preinj-2}
\tau(x)(g) = p_1(x(ga)) + p_2(x(gb)) + p_1(x(ga^{-1})) + p_2(x(gb^{-1}))
\end{equation*}
for all $x \in V^G$ and $g \in G$.  
Clearly $\tau$ is a linear cellular automaton admitting  $M = \{a,b,a^{-1},b^{-1}\}$ as a memory set.
We have $\tau(V^G) \subset (\K \times \{0\})^G \subsetneqq V^G$ so that  $\tau$ is not surjective.
\par
Let us show that $\tau$ is pre-injective. Suppose it is not. Then there exists a configuration 
$x_0 \in V^G$ with nonempty finite support $\Omega \subset G$ such that $\tau(x_0) = 0$.
Let $F$ denote the free subgroup generated by $a$ and $b$.
Choose a left coset $c_0 \in G/F$ such that $c_0$ meets $\Omega$.
The coset $c_0$ may be viewed as a regular tree of degree $4$ by joining two elements $g,h \in c_0$ if and only if $h^{-1}g \in M$.
Consider now an element $g_0 \in \Omega$ which is an ending point of the minimal tree spanned by $c_0 \cap \Omega$ in the tree $c_0$.
Observe that, among the four elements in $c_0$ which are adjacent to $g_0$, there are at least three elements outside $\Omega$.
As $g_0$ is in the support of $x_0$, we must have $p_1(x_0(g_0)) \not= 0$ or $p_2(x_0(g_0)) \not= 0$.
If $p_1(x_0(g_0)) \not= 0$, let us choose $g_1$ outside $\Omega$ such that $g_0 = g_1a$ or $g_0 = g_1a^{-1}$. This gives a contradiction since $g_0$ is then the only element in $\Omega$ which is adjacent to $g_1$ so that
\eqref{e:def-tau-surj-non-preinj-2} implies $\tau(x)(g_1) = p_1(x(g_0)) \not= 0$.
If $p_1(x_0(g_0)) = 0$ then $p_2(x_0(g_0)) \not= 0$. In this case, we choose an element $g_2$ outside $\Omega$ such that 
$g_0 = g_2b$ or $g_0 = g_2b^{-1}$. We then get $\tau(x_0(g_2)) = p_2(x_0(g_0)) \not= 0$ which yields also a contradiction. 
 \end{proof}

\begin{proposition}
\label{p:lca-pre-inj-not-surj-ls}
Let $G$ be a group and let $V$ be a one-dimensional vector space over a field $\K$.
Then:
\begin{enumerate}[\rm(i)]
\item
if $G$ is infinite, 
then there exist a linear subshift $X \subset V^G$  and a linear cellular automaton 
 $\tau\colon X \to X$ which is pre-injective but not surjective;
 \item
 if $G$ contains an infinite subgroup of infinite index,
then there exist an irreducible linear subshift $X \subset V^G$  and a linear cellular automaton 
 $\tau\colon X \to X$ which is pre-injective but not surjective;
 \item
  if $G$ is not locally finite (e.g. $G = \Z$),
then there exist a  linear subshift of finite type $X \subset V^G$  and a linear cellular automaton
 $\tau\colon X \to X$ which is pre-injective but not surjective;
 \item
 if $G$ contains an infinite finitely generated subgroup of infinite index (e.g. $G = \Z^2$),
then there exist an irreducible linear subshift of finite type $X \subset V^G$  and a linear cellular automaton 
 $\tau\colon X \to X$ which is pre-injective but not surjective.
 \end{enumerate}
 \end{proposition}

\begin{proof}
Suppose that $H$ is an infinite subgroup of $G$.
Consider the subset $X \subset V^G$ consisting of the configurations $x \in V^G$ which are constant on each left coset of $H$.
Clearly $X$ is a nonzero linear subshift of $V^G$.
The linear cellular automaton $\tau \colon X \to X$ defined by $\tau(x) = 0$ for all $x \in X$ is not surjective. However, $\tau$ is pre-injective.
Indeed, as every left coset of $H$ is infinite,
any two configurations in $X$ which  are almost equal  must coincide.
We obtain (i) by taking $H = G$.
\par
If $H$ is of infinite index in $G$,
we can find, for every finite subset $\Omega \subset G$, an element $g \in G$ so that no left coset of $H$ meets both $\Omega$ and $g\Omega$. This shows that $X$ is irreducible and
(ii) follows.
\par
If $H$ admits a finite generating subset $D \subset H$, then $X$ is of finite type since $X = X_G(D,L)$, where     $L \subset V^D$ denote the vector subspace of $V^D$ consisting of all   constant maps from $D$ to $V$. This shows (iii).
\par
 Finally, if $H$ is both finitely generated and of infinite index in $G$, then $X$ is an irreducible  linear subshift of finite type by the preceding observations. This gives (iv).
     \end{proof}

 Note that none of the linear subshifts $X \subset V^G$ appearing in the proof of Proposition 
\ref{p:lca-pre-inj-not-surj-ls}
is strongly irreducible ($G$ amenable or not).
Indeed, suppose that $\Delta$ is a finite subset of $G$.
Then, as $H$ is infinite, we can find an element $h_0 \in H$ which is not in $\Delta$.
The sets $\Omega_1 = \{h_0\}$ and $\Omega_2 = \{1_G\}$ satisfy $\Omega_1^{+\Delta} \cap \Omega_2 = \varnothing$. However, if $x_1 \in V^G$ is a nonzero constant configuration, we have $x_1 \in X$ but there is no configuration $x \in X$ which coincides with $x_1$ on $\Omega_1$ and with the zero configuration on $\Omega_2$.
 
\section{Surjective but not pre-injective linear cellular automata}
\label{sec:surj-but-not-pre-linj}

 In this section we describe examples of surjective but not pre-injective linear cellular automata 
$\tau \colon X \to X$, where 
 $G$ is a group, $V$ is a finite-dimensional vector space, and $X \subset V^G$ is a linear subshift.
 We recall that Theorem \ref{t;main} implies that there is no such example with $X$ strongly irreducible of finite type,
and in particular with $X = V^G$,
if the group $G$ is amenable.
When $G$ contains a nonabelian free subgroup and $\dim(V) = 2$,
one can construct a linear cellular automaton $\tau \colon V^G \to V^G$ which is 
 surjective but not pre-injective. This was done in \cite[Example 4.11]{semi-simple} for free groups of rank $2$ in a more general setting, namely for linear cellular automata whose alphabet is a module over any nonzero ring.

 \begin{proposition}
\label{p:lca-surj-not-pre-inj}
Let $G$ be a group and let $V$ be a $2$-dimensional vector space over a field $\K$.
Suppose that $G$ contains a nonabelian free subgroup (e.g. $G$ is a nonabelian free group). 
 Then there exists a linear cellular automaton $\tau\colon V^G \to V^G$ which is surjective but not 
 pre-injective.
 \end{proposition}

\begin{proof}
We may assume $V = \K^2$.
Let $q_1$ and $q_2$ be the endomorphisms of $V$ respectively defined by
$q_1(v) = (\lambda_1,0)$ and $q_2(v) = (0,\lambda_1)$ for all $v = (\lambda_1,\lambda_2) \in V$.
 Let $a$ and $b$ be two elements in $G$  generating a free subgroup of rank $2$.
 Consider the map $\tau \colon V^G \to V^G$ given by
 $$
\tau(x)(g) = q_1(x(ga)) + q_1(x(ga^{-1})) + q_2(x(gb)) + q_2(x(gb^{-1}))
$$
for all $x \in V^G$ and $g \in G$.
Clearly $\tau$ is a linear cellular automaton admitting  $M = \{a,b,a^{-1},b^{-1}\}$ as a memory set.
The configuration which takes the value $(0,1)$ at $1_G$ and is identically zero on 
$G \setminus \{1_G\}$ has nonempty finite support and is in the kernel of $\tau$. Therefore $\tau$ is not pre-injective.
\par
Let us show that $\tau$ is onto.
Let $z = (z_1,z_2) \in V^G$. We have to show the existence of a configuration $x= (x_1,x_2) \in V^G$ such that $z = \tau(x)$.
 Let $F$ denote the free subgroup of $G$ generated by $a$ and $b$.
 For $h \in F$, we denote by $\ell(h)$ the \emph{word length} of $h$, that is, the smallest integer 
 $n \geq 0$ such that $h$ can be written as a product $h = s_1s_2\cdots s_n$, where $s_i \in M$ for $1 \leq i \leq n$.
Let $R \subset G$ be a complete set of representatives for the left cosets of $F$ in $G$
so that every element $g \in G$ can be uniquely written in the form $g = rh$ with $r \in R$ and $h \in F$.
We define $x(g)$ by induction on $\ell(h)$. 
If $\ell(h) = 0$, that is, $g \in R$, we set $x(g) = (0,0)$.
If $\ell(h) = 1$, that is, $g = rs$ for some $r \in R$ and $s \in M$, we set
$$
x(g) =
\begin{cases}
(z_1(r),0) & \text{ if } s = a, \\
(z_2(r),0) & \text{ if } s = b, \\
(0,0) & \text{ if }  s = a^{-1} \text{ or } s = b^{-1}. 
\end{cases}
$$
Suppose now that, for some integer $n \geq 2$,
the value of $x$  has been defined at each element of the form  $ rh$, where $r \in R$ and $h \in F$ satisfies 
$\ell(h) \leq n -1$.
Let $g = rh$, where $r \in R$ and $h \in H$ satisfies $\ell(h) = n$.
Then $h$ can be uniquely written in the form $h = kss'$, where $k \in F$ satisfies $\ell(k) = n - 2$ and $s,s' \in M$ are such that $ss' \not= 1_G$. 
  We  set
  $$
  x(g) =
  \begin{cases}
  (z_1(rks) - x_1(rk),0) & \text{ if } s \in \{a,a^{-1}\} \text{ and } s = s', \\
  (z_2(rks),0) & \text{ if } s \in \{a,a^{-1}\} \text{ and } s' = b,  \\
  (z_1(rk), 0) & \text{ if }  s \in \{b,b^{-1}\} \text{ and  } s' = a, \\
  (z_2(rk) - x_2(rks), 0) & \text{ if }  s \in \{b,b^{-1}\} \text{ and  } s' = s, \\
  (0,0) & \text{ otherwise.}  \\
    \end{cases}
  $$
The configuration $x$ defined in this way clearly satisfies $z = \tau(x)$.
This shows that $\tau$ is surjective. 
\end{proof}

\bibliographystyle{siam}
\bibliography{lsft}

\end{document}